\documentclass{article}
\usepackage{amsmath}
\usepackage{amsthm}
\usepackage{amssymb}
\usepackage{enumerate}
\usepackage{hyperref}
\usepackage[utf8]{inputenc}
\usepackage[T1]{fontenc}

\newtheorem{thm}{Theorem}
\newtheorem{prop}[thm]{Proposition}
\newtheorem{lemma}[thm]{Lemma}
\newtheorem{claim}[thm]{Claim}
\newtheorem{conj}[thm]{Conjecture}
\newtheorem{question}[thm]{Question}
\newtheorem{fact}[thm]{Fact}
\newtheorem{observation}[thm]{Observation}

\def\Nat{{\mathbb{N}}}
\def\Real{{\mathbb{R}}}
\def\calS{{\mathcal{S}}}
\def\calB{{\mathcal{B}}}
\DeclareMathOperator{\calI}{\mathcal{I}}

\begin{document}
\title{Counterexamples to a conjecture of Harris on Hall ratio}

\author{
Adam Blumenthal\thanks{Iowa State University, Department of Mathematics, Iowa State University, Ames, IA. E-mail: {\tt ablument@iastate.edu}.}
\and
Bernard Lidický\thanks{Iowa State University, Department of Mathematics, Iowa State University, Ames, IA. E-mail: {\tt lidicky@iastate.edu}.
{Research of this author is partially supported by NSF grant DMS-1600390.}}
\and
Ryan R. Martin\thanks{Iowa State University, Department of Mathematics, Iowa State University, Ames, IA. E-mail: {\tt rymartin@iastate.edu}.
Research of this author is partially supported by a grant from the Simons Foundation (\#353292).
}
\and
Sergey Norin\thanks{Department of Mathematics and Statistics, McGill University, Montreal, Canada. E-mail: {\tt snorin@math.mcgill.ca}.
Research of this author is supported by an NSERC grant 418520.}
\and
Florian Pfender\thanks{Department of Mathematical and Statistical Sciences, University of Colorado Denver, E-mail: {\tt Florian.Pfender@ucdenver.edu}.
{Research of this author is partially supported by NSF grant DMS-1600483.}}
\and
Jan Volec\thanks{
Department of Mathematics, Faculty of Nuclear Sciences and Physical Engineering, Czech Technical University in Prague, Trojanova 13, 12000 Praha, Czech Republic. E-mail: {\tt jan@ucw.cz}.
Previous affiliations:
Faculty of Informatics, Masaryk University, Botanická 68A, 602 00 Brno, Czech Republic,
Department of Mathematics and Statistics, McGill University, Montreal, Canada,
where this author was supported by CRM-ISM fellowship, and
Department of Mathematics, Emory University, Atlanta, USA. This project has received funding
from the European Union’s Horizon 2020 research and innovation programme under the Marie Sk\l{}odowska-Curie grant
agreement No. 800607.
}
}

\date{}
\maketitle

\begin{abstract}
The Hall ratio of a graph $G$ is the maximum value of $v(H) / \alpha(H)$ taken
over all non-null subgraphs $H \subseteq G$. For any graph, the Hall ratio is a lower-bound
on its fractional chromatic number.
In this note, we present various constructions of graphs whose fractional
chromatic number grows much faster than their Hall ratio. This refutes a
conjecture of Harris.
\end{abstract}

\section{Introduction}
A graph $G$ is \emph{$k$-colorable} if its vertices can be colored with $k$ colors so
that adjacent vertices receive different colors.  The minimum integer $k$ such
that $G$ is $k$-colorable is called \emph{the chromatic number} of $G$, and it is
denoted by $\chi(G)$.

Various refinements and relaxations of the chromatic number have been considered in the
literature. One of the classical and most studied ones is the \emph{fractional
chromatic number}, which we denote by $\chi_f(G)$; see Section~\ref{sec:frac} for its definition.

A basic averaging argument reveals that $\chi_f(G) \ge v(G) / \alpha(G)$,
where $v(G)$ and $\alpha(G)$ are the number of vertices and the size
of a largest independent set in $G$, respectively.
Moreover, since $\chi_f(G) \ge \chi_f(H)$ for a subgraph $H \subseteq G$, it holds that
\[ \chi_f(G) \ge \frac{v(H)}{\alpha(H)} \quad \mbox{ for every non-null }H\subseteq G.\]
We define $\rho(G)$ --- the \emph{Hall ratio} of a graph $G$ --- to be the best
lower-bound obtained in this way, i.e.,
\[\rho(G) := \max_{\emptyset \neq H \subseteq G} \; \frac{v(H)}{\alpha(H)}
%= \max_{w: V \to \{0,1\}} \frac{\sum_{v\in V} w(v)}{\alpha(G,w)}
\,.\]

How tight is $\rho(G)$ as a lower bound for $\chi_f(G)$?
In~2009, Johnson~\cite{bib:johnson} suggested that there are graphs $G$ where
the value of $\chi_f(G) / \rho(G)$ is unbounded.
In earlier versions of~\cite{bib:harris} (see~\cite[Conjecture 6.2]{bib:harris:arXiv}),
Harris explicitly conjectured the opposite.
\begin{conj}
\label{conj:hall}
There exists $C$ such that $\chi_f(G) \le C\cdot\rho(G)$ for every graph~$G$.
\end{conj}
In 2016, Barnett~\cite{bib:barnett} constructed graphs showing
that if such a constant $C$ exists, then $C \ge 343 / 282 \sim 1.216$ improving an earlier bound $1.2$~\cite{bib:dhj}.
Our first result refutes Conjecture~\ref{conj:hall}.

\begin{thm}\label{thm:ce}
There exists $P_0$ such that for every $P \ge P_0$, there is a graph~$G$ with $\rho(G) \le P$ and $\chi_f(G) > P^2 / 33$.
\end{thm}

The proof of Theorem~\ref{thm:ce} is very short and simple, modulo some standard results about random graphs.
The following two theorems strengthen Theorem~\ref{thm:ce} at the expense of somewhat more technical proofs.

\begin{thm}\label{thm:k5free}
There exists $P_0$ such that for every $P \ge P_0$ there is a $K_5$-free graph $G$
with $\rho(G) \le P$ and $\chi_f(G) > P^2 / 82$.
\end{thm}

\begin{thm}\label{thm:recursive}
There exists $P_0$ such that for all $P \ge P_0$ there is a graph $G$ with $\rho(G) \le P$ and $\chi_f(G) \ge e^{\ln^2(P)/5}$.
\end{thm}

This note is organized as follows. In Section~\ref{sec:def}, we recall
definitions and properties of the fractional chromatic number, and
Erd\H{o}s-Rényi random graphs. Proofs of our results are in
Section~\ref{sec:results}. We conclude the note by Section~\ref{sec:outro}
with related open problems.

\section{Definitions and preliminaries}\label{sec:def}
The join of two graphs $G_1$ and $G_2$, which we denote by $G_1 \wedge G_2$,
is obtained by taking vertex-disjoint copies of $G_1$ and $G_2$, and adding all
the edges between $V(G_1)$ and $V(G_2)$.  More generally, for graphs
$G_1,G_2,\dots G_\ell$, we write $\bigwedge\limits_{i=1}^\ell G_i$ to denote
$\left(\bigwedge\limits_{i=1}^{\ell-1} G_i\right) \wedge G_\ell$. 

For a graph $H$ on the vertex-set $\{1,\dots,\ell\}$ and a collection
of $\ell$ vertex-disjoint graphs $G_1,\dots,G_\ell$, we define $H\{G_1,\dots,G_\ell\}$ to
be the graph obtained by taking a union $G_1,\dots, G_\ell$, and,
for every edge $ij \in E(H)$, adding all the edges between $V(G_i)$ and $V(G_j)$.
Note that if $G_1 \cong \dots \cong G_\ell$, then $H\{G_1,\dots,G_\ell\}$
corresponds to the composition (also known as the lexicographic product) of $G$
and $H$.  Also, observe that \[K_\ell\{G_1,\dots,G_\ell\} =
\bigwedge\limits_{i=1}^\ell G_i.\]

\subsection{Fractional chromatic number}\label{sec:frac}
We present a definition of the fractional chromatic number based on a linear
programming relaxation of an integer program computing the ordinary
chromatic number. For a graph $G$, let $\calI(G)$ be the set of all 
independent sets.  Let \textsc{fracc} be the following linear program.

\[
 \textsc{fracc} \begin{cases}
 \begin{array}{lll} 
   \text{Minimize }    & \displaystyle  \sum_{I \in \calI(G)} x_I \\
    \text{subject to }  & \displaystyle \sum_{\substack{I \in \calI(G)\\v\in I}} x_I \ge 1   &
    \quad \text{for  $v \in V(G)$;} \\
                               & x_I \ge 0  & \quad \text{for $I \in \calI(G)$.}
 \end{array}
 \end{cases}
\]
Furthermore, let \textsc{fracd} be the following program, which is the dual of \textsc{fracc}.
\[
 \textsc{fracd} \begin{cases}
 \begin{array}{lll} 
   \text{Maximize }    & \displaystyle  \sum_{v \in V(G)} y(v) \\
    \text{subject to }  & \displaystyle \sum_{\phantom{G}v \in I\phantom{(G}} y(v) \le 1 &
    \quad \text{for $I \in \calI(G)$;} \\
                               & y(v) \ge 0 & \quad \text{for $v \in V(G)$.}
 \end{array}
 \end{cases}
\]
Since these two linear programs are dual of each other, the LP-duality theorem
ensures that they have the same value, which we denote by $\chi_f(G)$.

Let us now mention a different way to introduce the fractional chromatic number.
As we have already mentioned, $\alpha(G) \ge v(G) / \chi_f(G)$.
Moreover, the lower-bound stays valid even in the setting where the vertices have
weights, and we measure the size of an independent set by the proportion of the weight it occupies
rather than its cardinality.

More precisely, let $G=(V,E)$ be a graph and $w: V \to \Real_+$ a weight function. Let $\alpha(G,w)$
be the maximum sum of the weights of the vertices that form an independent set, i.e.,
\[\alpha(G,w) := \max\limits_{I \in \calI} \sum_{v\in I} w(v)\,.\]
If we rescale an optimal solution of \textsc{fracc} by a factor $1/\chi_f(G)$
and interpret it  as a probability distribution on $\calI$, the linearity of expectation yields that
\[
\alpha(G,w) \ge 
\mathbb{E}_I\left[ \sum_{v \in I} w(v)\right] =
{\sum_{v \in V} w(v) \cdot \sum_{\substack{I \in \calI(G)\\v\in I}}  \frac{x_I}{\chi_f(G)}}
\ge \frac{\sum_{v \in V} w(v)}{\chi_f(G)}\,.
\]
On the other hand, any optimal solution of \textsc{fracd} yields a weight function $w_0$ for which the bound is tight, i.e., 
$\alpha(G,w_0) = {\sum_{v \in V} w_0(v)}/{\chi_f(G)}$. Therefore, \[
\chi_f(G) = \sup_{w: V \to [0,1]} \frac{\sum_{v\in V} w(v)}{\alpha(G,w)}\,.
\]
Note that the Hall ratio can be viewed as an integral version of the above, since
\[\rho(G) = \max_{w: V \to \{0,1\}} \frac{\sum_{v\in V} w(v)}{\alpha(G,w)}\,.\]
For other possible definitions of the fractional chromatic number, see~\cite{bib:ScUl97}.

We finish this section with a straightforward generalization of the fact that
the fractional chromatic number of the composition of two graphs is equal to
the product of their fractional chromatic numbers.
\begin{prop}
\label{prop:fracprod}
Let $H$ be a graph with the vertex-set $\{1,\dots,\ell\}$ and let $G_1,\dots,G_\ell$ be graphs.
It holds that $\chi_f(H) \cdot \min\limits_{i\in [\ell]} \chi_f(G_i) \le \chi_f(H\{G_1,\dots,G_\ell\})$.
In particular, $\chi_f\left( \bigwedge_{i=1}^\ell G_i \right) \ge \ell \cdot \min\limits_{i\in [\ell]} \chi_f(G_i)$.
\end{prop}

\begin{proof}
Without loss of generality, we may assume that $V(G_i) = \{1,\dots,v(G_i)\}$.
Let $w^H_1,\dots,w^H_\ell$ be any optimal solution of the dual program \textsc{fracd} for $H$,
and, for every $i\in[\ell]$, let $w^i_1,\dots,w^i_{v(G_i)}$, be any optimal solution of \textsc{fracd}
for $G_i$.

Let $G:=H\{G_1,\dots,G_\ell\}$. For a vertex $(i,j) \in V\left(G\right)$, where $i \in [\ell]$ and $j \in [v(G_i)]$,
we set $y_{i,j}:=w^H_i \cdot w^i_j$. It holds that
\[
\sum_{(i,j) \in V(G)} y_{i,j} = \sum_{i \in [\ell]} w^H_i \cdot \sum_{j \in V(G_i)} w^i_j =
\sum_{i \in [\ell]} w^H_i \cdot \chi_f(G_i) \ge \chi_f(H) \cdot \min\limits_{i\in [\ell]} \chi_f(G_i)\,.
\]
We claim that $\left(y_{i,j}\right)$, where $(i,j)\in V(G)$, is a feasible solution of \textsc{fracd} for $G$.

Indeed, fix any $I \in \calI(G)$. For $i \in [\ell]$, let $I_i := \{ j\in[v(G_i)]: (i,j) \in I \}$. Since $I_i \in \calI(G_i)$, it holds that
\[\sum\limits_{j \in I_i} w^H_i \cdot w^i_j=w^H_i \cdot \sum\limits_{j \in I_i} w^i_j \le w^H_i\,.\]
On the other hand, the set $I_H := \{i \in [\ell]: \exists (i,j) \in I\}$ is independent in $H$. Therefore,
\[
\sum_{(i,j) \in I} y_{i,j} = \sum_{i \in I_H} \sum_{j \in I_i} y_{i,j} = \sum_{i \in I_H} w^H_i \cdot \sum_{j \in I_i} w^i_j \le \sum_{i \in I_H} w^H_i \le 1\,.
\]
\end{proof}
We note that a similar composition of optimal solutions of \textsc{fracc} yields
$\chi_f(H\{G_1,\dots,G_\ell\}) \le \chi_f(H) \cdot \max\limits_{i \in [\ell]} \chi_f(G_i)$,
but we will never need this bound. However, we will use the following analogue of this
bound for proper colorings.

\begin{prop}
\label{prop:col}
Let $H$ be a graph with the vertex-set $\{1,\dots,\ell\}$ and let $G_1,\dots,G_\ell$ be graphs.
It holds that $\chi(H) \cdot \max\limits_{i\in [\ell]} \chi(G_i) \ge \chi(H\{G_1,\dots,G_\ell\})$.
\end{prop}
\begin{proof}
Let $k:=\max_{i\in [\ell]} \chi(G_i)$, and $d$ be a proper $\chi(H)$-coloring of $H$.
Next, for every $i \in [\ell]$, let $c_i: V(G_i) \to [k]$ be a proper $k$-coloring of $G_i$.
 It is straightforward to verify that assigning each vertex $v \in V_i$ a color $\left(d(i),c_i(v)\right)$
yields a proper coloring of $\chi\left(H\{G_1,\dots,G_\ell\}\right)$ using $\chi(H) \cdot k$ colors.
\end{proof}

Finally, the following observation is going to be useful in the next section.
\begin{observation}
\label{obs:simple}
If every $H\subseteq G$ has at most $|V(H)|$ edges, then $\chi(G)\le 3$.
\end{observation}
\begin{proof}
Without loss of generality, we may assume $G$ is connected. Since $|E(G)| \le |V(G)|$, the graph $G$
contains at most one cycle and hence it is $3$-colorable.
\end{proof}

\subsection{Sparse Erd\H{o}s-Rényi random graphs}\label{sec:GnD}
Let $G_{n,p}$ be a random graph on $\{1,2,\dots,n\}$ where each pair of vertices
forms an edge independently with probability $p$.
We now recall some well-known properties of~$G_{n,\frac Dn}$ we are going to use.

\begin{prop}\label{prop:GnD}
There exists $C_0$ such that for every $C \ge C_0$ the following is true:
There exists $n_0=n_0(C) \in \Nat$ such that for every $n \ge n_0$ there is
an~$n$-vertex triangle-free graph $G=G^1(n,C)$ with the following properties:
\begin{enumerate}[(A)]
\item $1.001 \cdot C > \chi(G) \ge \chi_f(G) \ge \frac n{\alpha(G)} > C$,
and
\item for all $k\le \sqrt{\ln n}$, every $k$-vertex subgraph of $G$ is $3$-colorable. \label{prop:B}
\end{enumerate}
\end{prop}
\begin{proof}
Suppose that $C$ and $n$ are sufficiently large, and let $D > 1$ be such that $C=\frac D{2\cdot\ln D}$.
By~\cite{bib:frieze} and~\cite{bib:luczak}, a random graph $G_{n,\frac Dn}$
satisfies with high probability  $\alpha(G_{n,\frac Dn}) < n/C$ and $\chi(G_{n,\frac Dn}) < 1.001 \cdot C$, respectively.

Next, the expected number of subgraphs $H$ in $G_{n,\frac Dn}$ with $v(H) \le \sqrt{\ln n}$ and more than $v(H)$ edges
is at most
\[
\sum\limits_{k=3}^{\sqrt{\ln n}}
2^{k^2} \cdot n^k \cdot \left(\frac Dn\right)^{k+1}
\le \sqrt{\ln n} \cdot \frac{D^{\sqrt{\ln n}+1}}{n^{1-\ln 2}}
= O\left(n^{-0.3}\right).
\]
By Markov's inequality, $G_{n,\frac Dn}$ has no such $H$ with high probability,
hence the property (\ref{prop:B}) follows from Observation~\ref{obs:simple}.

Finally, Schürger~\cite{bib:schurger} showed that the number of triangles in
$G_{n,\frac Dn}$ converges to the Poisson distribution with mean $\Theta\left(D^3\right)$.
Therefore, $G_{n,\frac Dn}$ is triangle-free with probability $e^{-\Theta\left(D^3\right)} > 0$.
Note that a similar estimate can also be deduced using the FKG inequality.
\end{proof}

\section{Counter-examples to Conjecture~\ref{conj:hall}}\label{sec:results}

We start with a simple construction of a sequence of graphs for which $\chi_f(G)\gg\rho(G)$.
Each graph $G$ is the join of the graphs $G^1(n_i,C)$ of very different orders.

\begin{proof}[Proof of Theorem~\ref{thm:ce}]
Let $C_0$ be the constant from Proposition~\ref{prop:GnD}, and $P_0 := 8C_0$.

Given $P\ge P_0$, let $\ell := \lfloor P/4 \rfloor$, $C:=P/8$, and $n_1:=n_0(C)$ from Proposition~\ref{prop:GnD}.
For all $j \in [\ell-1]$, let $n_{j+1} := \left\lceil e^{2 \cdot n_j^2}\right\rceil$, and, for all $i \in [\ell]$,
let $G_i := G^1(n_i,C)$.
We set $G := \bigwedge\limits_{i=1}^\ell G_i$.

By Proposition~\ref{prop:fracprod}, $\chi_f(G) > \ell \cdot C > P^2/33$. It only remains to
prove that $\rho(G) \le P$, i.e., that $\alpha(G[X]) \ge v(G[X])/P$ for every $X \subseteq V(G)$.

Fix $X\subseteq V(G)$, and let $X_i:=V(G_i) \cap X$ for $i \in [\ell]$.
We split the indices  into two categories,
\emph{small} and \emph{big}, based on $|X_i|$ with respect to $v(G_i)=n_i$.
Specifically, let
\[
\calS := \left\{ i \in [\ell] : |X_i| < \sqrt{\ln n_i} \right\},
\quad
\mbox{and}
\quad
\calB := [\ell] \setminus \calS.
\]
Next, let $H_S$ and $H_B$ be the subgraphs of $G$ induced by $\bigcup\limits_{i \in \calS} X_i$
and $\bigcup\limits_{i \in \calB} X_i$, respectively, and $v_s$ and $v_b$ their respective orders.
In both of these subgraphs, we can find quite large independent sets.
\begin{claim}\label{cl:ce1}
$H_S$ has an independent set of size at least ${4v_s}/{3P}$.
\end{claim}
\begin{proof}
Fix $i \in \calS$ such that $|X_i|$ is maximized. Note that $|X_i| \ge {v_s}/{|\calS|}$.
The property (\ref{prop:B}) of $G_i$ established in Proposition~\ref{prop:GnD}
yields that $G[X_i]$ is $3$-colorable, and hence its largest color class has size at least
\[\frac{v_s}{3 |\calS|} \ge \frac{v_s}{3\ell} \ge \frac{4v_s}{3P}\,,\]
which finishes the proof.
\end{proof}

\begin{claim}\label{cl:ce2}
$H_B$ has an independent set of size at least ${4v_b}/P$.
\end{claim}
\begin{proof}
Let $m$ be the largest element of $\calB$. Since $G_m$ is $(0.51\ell)$-colorable, $G[X_m]$ contains an independent
set of size at least $1.9 \cdot |X_m|/\ell$.
If $m=1$, then $|X_m| = v_b$. On the other hand, if $m \ge 2$, then 
\[
1.9 \cdot |X_m| \ge |X_m| + 0.9 \cdot \sqrt{\ln n_m} > |X_m| + 1.2 \cdot n_{m-1} > |X_m| + \sum\limits_{i=1}^{m-1} n_i \ge v_b.
\]
We conclude that $H_B$ has an independent set of size at least $v_b/\ell \ge 4v_b/P$.
\end{proof}

If $v_s \ge 3 |X| /4$, then we find an independent set of size at least $|X|/P$ in $H_S$ by Claim~\ref{cl:ce1}.
Otherwise, $v_b \ge |X| / 4$, and Claim~\ref{cl:ce2} guarantees an~independent set in $H_B$ of size at least $|X|/P$.
\end{proof}

\subsection{$K_5$-free and iterated constructions}

As we have already noted in Section~\ref{sec:def}, the graph $G=\bigwedge\limits_{i=1}^\ell G_i$ constructed in Theorem~\ref{thm:ce}
can be equivalently viewed as $K_\ell\{G_1,G_2,\dots,G_\ell\}$. An adaptation of the proof
of Theorem~\ref{thm:ce} will show that replacing $K_\ell$ by a graph from Proposition~\ref{prop:GnD}
yields another graph $G^2$ with $\chi_f\left(G^2\right) \sim \left(\rho\left(G^2\right)\right)^2$.
However, as all the graphs involved in the composition are now triangle-free, $G^2$ will be $K_5$-free.

But we do not need to stop here. Since we have now much better control on~the~chromatic numbers of small subgraphs in $G^2$
than in the original graph~$G$, replacing the graphs $G_i=G^1(n_i,C)$ in the composition by $n_i$-vertex
variants of $G^2$ yields a graph $G^3$ with $\chi_f\left(G^3\right) \sim \left(\rho\left(G^3\right)\right)^3$.
Repeating this procedure $k$-times leads to a construction of a graph~$G^{k+1}$ with $\chi_f\left(G^{k+1}\right) \sim \left(\rho\left(G^{k+1}\right)\right)^{k+1}$.

In order to present our proofs of Theorems~\ref{thm:k5free} and~\ref{thm:recursive},
we need to introduce some additional notation.  Let us start with recalling the
Knuth's up-arrow notation
\[
  a\uparrow^{(k)} b=
   \begin{cases}
    a^b & \text{if }k=1, \\
    1 & \text{if }k\ge 1\text{ and }b=0,\\
    a\uparrow^{(k-1)}(a\uparrow^{(k)}(b-1)) & \text{otherwise,}
   \end{cases}
\]
where $a,b,k \in \Nat$,
and its inverse $a\downarrow^{(k)} n$, which is the largest integer $b$ such that $n \ge a\uparrow^{(k)} b$.
Using this, we define the following Ackermann-type
function $F_k(b)$ and its inverse $f_k(b)$:
\[
F_k(b):= 2 \uparrow^{(k)} b \quad\quad \text{and} \quad\quad f_k(b) := 2
\downarrow^{(k)} b
.
\] 
Note that $F_1(b) = 2^b$ and $f_1(b) = \left\lfloor \log_2(b) \right\rfloor$,
and  for every $k \in \Nat$ it holds that $F_k(1)=2$ and $F_k(2)=4$. The functions also satisfy the following
properties:
\begin{fact}\label{fact}
For every $k \in \Nat$, the following holds:
\begin{enumerate}
\item $f_{k}(f_{k}(F_{k+1}(n+2)))=F_{k+1}(n)$ for every $n\in\Nat$,

\item $f_{k+1}(4M) < f_{k}(f_{k}(M))$ for every $M \ge F_k(F_k(7))$, and

\item $\sum_{b=0}^n F_k(b) < F_k(n+1)$ for every $n\in\Nat$.

\end{enumerate}
\end{fact}
\noindent For a proof, see Appendix~\ref{app}. We are now ready to present the main lemma.

\begin{lemma}\label{lem:recursive}
Let $C_0$ be the constant from Proposition~\ref{prop:GnD}.
For every $k \in \Nat$ and $C \ge C_0$ there is $n_0 := n_0(k,C)$ such that for all $n \ge n_0$ there
is an $n$-vertex $K_{2^k+1}$-free graph $G:=G^{k}(n,C)$ with the following properties:
\begin{itemize}
\item $\chi_f(G) \ge C^k$,
\item $\rho(G) \le 1.001 \cdot 3^{k} \cdot C$, and
\item $G[W]$ is $3^k$-colorable for every $W \subseteq V(G)$ such that $|W| \le f_k(f_k(n))$.
\end{itemize}
\end{lemma}
\begin{proof}
For any fixed $C \ge C_0$, we proceed by induction on $k$. As the case $k=1$
follows by letting $G:=G^1(n,C)$ from Proposition~\ref{prop:GnD}, we may assume $k\ge2$.

Let $M$ be the smallest positive integer such that $f_k(4M) \le f_{k-1}\left(f_{k-1}(M)\right)$.
Note that $M \le F_{k-1}(F_{k-1}(7))$ by the second property of Fact~\ref{fact}.
We set $n_0(k,C) := \max\left\{M,F_k\left(4 \cdot n_0(k-1,C)\right)\right\}$.
Given $n \ge n_0(k,C)$, we define $m$ to be the largest integer such that \[
m + \sum\limits_{i=2}^{m} F_k (m+3i-6) \le n.
\]
Note that $F_k(4m-1)>n$, as otherwise the third property of Fact~\ref{fact} yields
\[\left(m+1\right) + \sum\limits_{i=2}^{m+1} F_k (m+3i-6) \le F_k(4m-1) \le n,\]
contradicting the maximality of $m$.
Therefore, \[F_k(4m) > F_k(4m-1) > n_0(k,C) \ge F_k(4 n_0(k-1,C)),\] and hence $m > n_0(k-1,C)$.
We set $b_1 := m$, and $b_i := F_k(m+3i-6)$ for every $i=2,3,\dots,m-1$. Finally, we set $b_m := n - \sum_{i<m} b_i$.

Let $H:=G^1(m,C)$, and $G_i:=G^{k-1}(b_i,C)$ for all $i \in [m]$.
We define $G:= H\{G_1,G_2,\dots,G_m\}$. Clearly, the graph $G$ contains no $K_{2^{k}+1}$.
In the following three claims, we show that $G$ has the desired three properties:

\begin{claim}
$\chi_f(G) \ge C^k$. 
\end{claim}
\begin{proof}
By the induction hypothesis, $\chi_f(H) \ge C$ and $\chi_f(G_i) \ge C^{k-1}$ for all $i \in [m]$.
Therefore, Proposition~\ref{prop:fracprod} yields the desired lower-bound on $\chi_f(G)$.
\end{proof}

\begin{claim}
$\rho(G) \le 1.001 \cdot 3^{k} \cdot C$.
\end{claim}
\begin{proof}
Fix an $X \subseteq V(G)$. Our aim is to show that $\alpha\left(G[X]\right) \ge |X| / \left( 1.001 \cdot 3^{k} \cdot C\right)$.
For $i\in [m]$, let $X_i$  be $X \cap V(G_i)$. As in the proof of Theorem~\ref{thm:ce}, let
\[
\calS := \big\{i \in [m] : |X_i| \le f_{k-1}(f_{k-1}(b_i))\big\} \quad \mbox{ and } \quad \calB := [m] \setminus \calS .
\]

First, suppose the case $\left|\bigcup_{i \in \calS} X_i\right| \ge |X|/3$.
By the definition of $\calS$ and the properties of $G_i$,
every subgraph $G[X_i]$, where $i \in \calS$, has an independent set of size at least $|X_i|/3^{k-1}$. On the other hand, $\chi(H) < 1.001 \cdot C$,
so the projection of at least one of the color classes of the optimal coloring of $H$ on $\bigcup_{i \in \calS} X_i$
contains an independent set of size at least
\[
\sum\limits_{i \in \cal S} \frac{|X_i|}{3^{k-1}} \cdot \frac 1{1.001 \cdot C} \ge \frac{|X|}{1.001 \cdot 3^{k} \cdot C}\,.
\]

Now suppose $\left|\bigcup_{i \in \calB} X_i\right| \ge 2|X|/3$, and let $z$ be the maximum index in $\calB$.
If $z=1$, then $|X_1| \ge 2|X|/3$. On the other hand, if $z\ge2$, then
\[
f_{k-1}(f_{k-1}(b_z)) \ge f_{k-1}(f_{k-1}(F_k(m+3z-6))) =
F_k(m+3z-8) \ge \sum\limits_{i < z} b_i ,
\]
where the equality and the last inequality follow from the first and the third property of Fact~\ref{fact}, respectively.
Therefore, $|X_z| \ge {|X|}/3$ and $G_z[X_z]$ contains an independent set of the sought size by $\rho(G_z) \le 1.001 \cdot 3^{k-1} \cdot C$.
\end{proof}

\begin{claim}
$G[W]$ is $3^k$-colorable for every $W \subseteq V$ with $|W| \le f_k(f_k(n))$.
\end{claim}
Fix a set $W \subseteq V$ of size at most $f_k(f_k(n))$. Firstly, let $Z := \{i: W \cap V(G_i) \neq \emptyset\}$.
Clearly, $|Z| \le |W| \le f_k(f_k(n))$.  Since $f_k(n) \le 4m$ and $f_k(x) \ll \log_2\log_2(x/4)$, 
we conclude that $|Z| \le \log_2\log_2(v(H))$. Therefore, there exists a proper $3$-coloring of the induced subgraph $H[Z]$.

By the second property of Fact~\ref{fact}, for every $i \in [m]$ it holds that
\[
|V(G_i) \cap W| \le |W| \le f_k(4m) \le f_{k-1}(f_{k-1}(m)) \le f_{k-1}(f_{k-1}(b_i)).
\]
Therefore, the induction hypothesis yields that each $V(G_i) \cap W$ induces a $3^{k-1}$-colorable
subgraph of $G$, and hence $\chi(G[W]) \le 3^k$ by Proposition~\ref{prop:col}. 
\end{proof}

Theorem~\ref{thm:k5free} is a direct consequence of Lemma~\ref{lem:recursive} applied with $k=2$.
It remains to establish Theorem~\ref{thm:recursive}:

\begin{proof}[Proof of Theorem~\ref{thm:recursive}]
Let $P_0 := \left(2C_0\right)^2$. Given $P \ge P_0$, let $C:= \sqrt{P/1.001}$  and $k:=\lfloor \log_3 C \rfloor$. 
Applying Lemma~\ref{lem:recursive} with $k$ and $C$ yields an $n_0(k,C)$-vertex graph $G$ with
$\rho(G) \le P$ and
$
\chi_f(G) \ge C^{\lfloor \log_3 C \rfloor} > e^{0.9 \cdot \ln^2(C)} > e^{\ln^2(P)/5}
$.
\end{proof}

\section{Concluding remarks}\label{sec:outro}

We presented various constructions of graphs where the
fractional chromatic number grows much faster than the Hall ratio, which
refuted Conjecture~\ref{conj:hall}.  It is natural to ask whether the
conclusion in Conjecture~\ref{conj:hall} can be relaxed and the fractional
chromatic number of a graph is always upper-bounded by some function of its
Hall ratio.
\begin{question}
\label{q:bounded}
Is there a function $g:\Real \to \Real$ such that $\chi_f(G) \le g(\rho(G))$ for~every graph $G$?
\end{question}
Theorem~\ref{thm:recursive} shows that if such a function $g$ exists, then
$g(x) \ge e^{\ln^2(x)/5}$.
While preparing our manuscript, we have learned that Dvořák, Ossona de Mendez and Wu~\cite{bib:dow} constructed graphs with
Hall ratio at most $18$ and arbitrary large fractional chromatic number. Therefore,
the answer to Question~\ref{q:bounded} is no.

Conjecture~\ref{conj:hall} was partially motivated by another conjecture of
Harris concerned with fractional colorings of triangle-free graphs, which was
inspired by a famous result of Johansson~\cite{bib:johansson} (for a recent
short proof, see~\cite{bib:molloy}) stating that
$\chi(G) = O(\Delta/\ln \Delta)$ for every triangle-free graph $G$ with maximum degree
$\Delta$.
\begin{conj}[{\cite[Conjecture 6.2]{bib:harris}}]
\label{conj:frac}
There is $C$ such that $\chi_f(G) \le C \cdot d/\ln d$ for every triangle-free $d$-degenerate graph $G$. 
\end{conj}
A classical result of Ajtai, Koml\'os, and Szemer\'edi~\cite{bib:aks} together
with an averaging argument yield that $\rho(G) = O(d/\ln d)$ for $G$ and $d$ as above.
Therefore, if Conjecture~\ref{conj:hall} could be recovered in the
triangle-free setting, it would immediately yield the sought bound on $\chi_f$
in Conjecture~\ref{conj:frac}. 
\begin{question} \label{q:trgfree}
Is there $C$ such that $\chi_f(G) \le C\cdot\rho(G)$ for every triangle-free graph~$G$?
\end{question}
In~\cite{bib:johnson}, it has been mentioned that the sequence of Mycelski graphs might provide a negative
answer to Question~\ref{q:trgfree}, but we still do not know.
For $K_5$-free graphs, Theorem~\ref{thm:k5free} shows that the answer is
definitely negative. As a possibly simpler question, does the answer stay negative in case of $K_4$-free graphs?
\begin{question}
\label{q:k4free}
Is there $C$ such that $\chi_f(G) \le C\cdot\rho(G)$ for every $K_4$-free graph~$G$?
\end{question}

Let us conclude with an additional motivation for studying Conjecture~\ref{conj:frac}.
Esperet, Kang and Thomass\'e~\cite{bib:ekt} conjectured that 
dense triangle-free graphs must contain dense induced bipartite subgraphs.
\begin{conj}[{\cite[Conjecture 1.5]{bib:ekt}}]\label{conj:bip}
There exists $C>0$ such that any triangle-free graph with minimum degree at least
$d$ contains an induced bipartite subgraph of minimum degree at least $C \cdot \ln d$.
\end{conj}
Erd\H{o}s-Rényi random graphs of the appropriate density show that the bound would be, up to the constant $C$, best possible.
A relation between the fractional chromatic number and induced bipartite
subgraphs proven in~\cite[Theorem 3.1]{bib:ekt} shows that if Conjecture~\ref{conj:frac}
holds, then Conjecture~\ref{conj:bip} holds as well.
Very recently, Kwan, Letzter, Sudakov and Tran~\cite{bib:klst} proved a slightly 
weaker version of Conjecture~\ref{conj:bip} where the bound $C \cdot \ln n$ is
replaced by $C \cdot \ln n / \ln\ln n$.

\section*{Acknowledgments}
The authors thank Zdeněk Dvořák for sharing the manuscript of~\cite{bib:dow} with them,
and Józsi Balogh, Sasha Kostochka and Yani Pehova for fruitful discussions at the beginning
of the project.
The authors also thank to the anonymous referees for carefully reading the
manuscript and for their valuable comments, which greatly improved the
presentation of our results.

This work was partially supported by NSF-DMS grant \#1604458
``Collaborative Research: Rocky Mountain Great Plains Graduate Research Workshops in Combinatorics''
and by NSA grant H98230-18-1-0017
``The 2018 and 2019 Rocky Mountain Great Plains Graduate Research Workshop in Combinatorics''.

\appendix
\section{Proof of Fact~\ref{fact}}\label{app}
The definitions of $f_k$ and $F_{k+1}$ readily yield that $f_k\left(F_{k+1}(n+1)\right) = F_{k+1}(n)$.
Therefore, $f_k\left(f_k\left(F_{k+1}(n+2)\right)\right) = F_{k+1}(n)$ proving the first property.

%%, $F_{k+1}(3) = F_{k}(4)$. 
For every $k,n \in \Nat$, a straightforward induction yields that $F_k(n) \ge n+1$. This
in turn implies that $F_{k+1}(n) = F_k(F_{k+1}(n-1)) \ge F_k(n) \ge 2^n$.
Similarly, for all $k\in \Nat$, the functions $F_k(\cdot)$ and $f_k(\cdot)$ are monotone non-decreasing.
Therefore, 
for all $k \in \Nat$  and $n \ge 7$, it holds that
\[
F_{k+1}(n) = F_k(F_k(F_k(F_k(F_{k+1}(n-4))))) \ge 2^{F_k\left(F_k\left(2^{n-4}\right)+1\right)}
\ge 4 \cdot F_k(F_k\left(n+1\right)+1) 
.
\]
Since $F_k(f_k(M)+1) > M \ge F_k(f_k(M))$, we assert that
$f_{k+1}(4M) < f_k(f_k(M))$ for all $M \ge F_k(F_k(7))$.
Indeed, as otherwise 
\[
4M \ge F_{k+1}(f_{k+1}(4M)) \ge F_{k+1}(f_k(f_k(M))) \ge 4\cdot F_k(F_k\left(f_k(f_k(M))+1\right)+1) > 4M
,\]
a contradiction. This concludes the proof of the second property.

The last property is proven by induction on $k$.
Indeed, the case $k=1$ is the sum of a geometric progression. If $k\ge2$, then by induction hypothesis
\[
\sum_{b=0}^n F_{k+1}(b) = \sum_{b=0}^n F_k(F_{k+1}(b-1)) 
\le
\sum_{i=0}^{F_{k+1}(n-1)} F_k(i) < F_k(F_{k+1}(n-1)+1)
.
\]
However, the right-hand side is at most $F_k(F_k(F_{k+1}(n-1))) = F_{k+1}(n+1)$.

\end{document}